\documentclass[leqno]{amsart}
\usepackage{amsmath}
\usepackage{amssymb}
\usepackage{amsthm}
\usepackage{enumerate}
\usepackage[mathscr]{eucal}
\theoremstyle{plain}
\newtheorem{theorem}{Theorem}[section]

\newtheorem{lemma}{Lemma}[section]

\theoremstyle{definition}

\usepackage[pagewise]{lineno}
\begin{document}

\title[Farthest point problem and M-compact sets]{Farthest point problem and M-compact sets}
\author[Debmalya Sain, Kallol Paul \and Anubhab Ray]{Debmalya Sain, Kallol Paul \and Anubhab Ray}

\newcommand{\acr}{\newline\indent}

\address{\llap{\,}Department of Mathematics\acr
                              Jadavpur University\acr
                              Kolkata 700032\acr
                              West Bengal\acr
                              INDIA}
\email{saindebmalya@gmail.com; kalloldada@gmail.com; anubhab.jumath@gmail.com}

\thanks{First author would like to thank UGC, Govt. of India and third author would like to thank DST, Govt. of India  for the financial support.} 

\subjclass[2010]{Primary 46B20, Secondary 46B99}
\keywords{Uniquely remotal ; M-compact; strictly convex}

\begin{abstract}
In this paper we give an elementary proof of the fact that every uniquely remotal set is singleton  in a finite dimensional strictly convex normed linear space. We show that if A is a uniquely remotal M-compact subset with $ A'\neq \emptyset $ then $A'$ is M-compact and uniquely remotal. We also show that if A is a uniquely remotal M-compact set and $A'$ is compact then A is singleton. 

\end{abstract}

\maketitle

\section{Introduction.}
Let $ (\mathbb{X},\| .\|) $ be a normed linear space and $ A $ be a nonempty bounded subset of $ \mathbb{X} $. For any $ x \in \mathbb{X},$  the farthest distance of $x$ from the set $A$ is denoted by $ D(x,A) $ i.e.,
\[ D(x,A) = \sup \limits_{a \in A}  \|x-a\| .\]
The farthest distance of $x$ from $A$ may or may not be attained by some element of $A,$ if the distance is attained then the collection of all such points of $A$ is denoted by  $ F(x,A) $ i.e., 
\[ F(x,A) = \{ a \in A : \|x-a\| = D(x,A)\} .\]
The collection of all points in $A$ for which the farthest distance  of $x$ from $A$ is attained for some $ x \in \mathbb{X}$ is denoted by $Far(A) $ i.e.,
   \[ Far{A} = \{ a \in A : ~~a \in F(x,A)  ~~\mbox{for some}~~ x \in \mathbb{X} \} .\]
		
We say that $ A $ is remotal  if $ F(x,A) $ is non-empty  for each $ x \in \mathbb{X} $ and $A$ is said to be uniquely remotal if $F(x,A)$ is singleton for each $x \in \mathbb{X}.$ \\

 The Farthest Point Problem (FPP) can now be stated as follows :\\

  \textbf{Must every uniquely remotal set in a Banach space be singleton?} \\

The FPP was proposed by Motzkin, Starus and Valentine in \cite{MSV}, in context of the Euclidean space $ E^{n}.$ The problem was considered in the setting of Banach spaces by Klee in \cite{K}, where he proved that every compact uniquely remotal subset of a Banach space is a singleton. In \cite{A}, Asplund solved the FPP in the affirmation in any finite dimensional Banach space with respect to a norm which is not necessarily symmetric. The method employed by the Asplund to prove this profound result involves convexity calculus and is quite astounding. In this paper we wish to give a simple  proof of the fact that in finite dimensional strictly convex Banach spaces every uniquely remotal set is a singleton.

The notion of M-compactness in normed linear spaces plays a significant role in the study of the FPP. $ A \subseteq X $ is said to be M-compact if every maximizing sequence in $ A $ is compact. A sequence $ \{a_n\} $ in $ A $ is said to be maximizing if for some $ x \in \mathbb{X} $, $ \|x-a_n\| \rightarrow D(x,A) $. It is easy to observe that M-compactness is a proper generalization of the usual compactness and M-compact sets may not be closed in a Banach space. M-compact sets in Banach spaces were considered by Vlasov in \cite{V}. Panda and Kapoor utilized the concept of M-compactness in the study of FPP in \cite{PK}.

In this paper we aim to further explore the importance of the notion of M-compactness in the study of the FPP. In fact we prove several results to illustrate the vital role played by the notion of M-compactness in determining the metric geometry of the underlying normed linear space in the context of the FPP. It is easy to see that \cite{PK} if $ A $ is M-compact then $ \overline{A} $ is M-compact but converse is not true. In this paper we show that if in addition $ A $ is uniquely remotal then converse holds in strictly convex normed linear spaces i.e., if $ \overline{A} $ is M-compact then $A$ is M-compact in a strictly convex normed linear space. In this paper, we prove that in a normed linear space $ \mathbb{X}, $ if $ A $ is uniquely remotal and M-compact subset of $ \mathbb{X}, $ then the derived set $ A^{'}, $ is either empty or uniquely remotal and M-compact. Panda and Kapoor proved in \cite{PK} that in a normed linear space admitting centers, every uniquely remotal M-compact set is a singleton. Here we proved that if $ A $ is a uniquely remotal and M-compact subset of a Banach space $ \mathbb{X} $ and $ A^{'} $ is compact then $ A $ is a singleton.

\section{ Main Results.}

In this section we first show that in a strictly convex normed linear space, closure of a uniquely remotal set is uniquely remotal and use it to give an easy proof of the fact that every uniquely remotal set in a finite dimensional strictly convex normed linear space is singleton. It is still unknown whether closure of a uniquely remotal set in any normed linear space is uniquely remotal.

\begin{theorem}\label{theorem:closure}
Let $ X $ be a strictly convex normed linear space. If $ A \subseteq X $ is uniquely remotal then $ \overline{A} $ is also uniquely remotal.
\end{theorem}
\begin{proof} As $ A $ is uniquely remotal and we know that, for any $ x \in X $, $ D(x,A) = D(x, \overline{A}) $, so for any $ x \in X $, there exists $ a \in A $ such that $ \|x-a\| = D(x,\overline{A}) $ and $ a \in A \subseteq \overline{A} $. So, $ \overline{A} $  is remotal. Suppose, $ \overline{A} $ is not uniquely remotal, then there exists $ x_0 \in X $ such that, $ F(x_0, \overline{A}) $ is not singleton, i.e. there exists $ w_0 \in \overline{A} \setminus{A} $ and unique $ a_{x_0} \in A $ such that $ \|x_0-w_0\| = \sup\limits_{a \in \overline{A}}\|x_0-a\| = \sup\limits_{a'\in A}\|x_0-a'\| = \|x_0-a_{x_0}\|$. Now, choose $ 1<t_0<2 $ and fixed after choice. Consider $ z_0 = (1-t_0)w_0 + t_0x_0 $. 
Now, for any $ a \in A \setminus\{a_{x_0}\} $,
\begin{eqnarray*}
        \|z_0-a\| & = & \|z_0-x_0 + x_0 - a\| \\
				          & \leq & \|z_0-x_0\| + \|x_0-a\| \\
									& = & \|(1-t_0)w_0-t_0x_0-x_0\| + \|x_0-a\| \\
									& < & (t_0-1)\|x_0-w_0\| + \|x_0-w_0\| \\
									& = & t_0\|x_0-w_0\| \\
									& = & \|z_0-w_0\| \ldots \ldots (1) 
\end{eqnarray*}
\noindent Now, if $ z_0 $, $ x_0 $ and $ a_{x_0} $ are non co-linear then $ (z_0-x_0) $ and $ (x_0-a_{x_0}) $ are linearly independent. 
Then, $ \|z_0-a_{x_0}\| = \|z_0-x_0+x_0-a_{x_0}\| < \|z_0-x_0\| + \|x_0-a_{x_0}\| $ ( as $ X $ is strictly convex space ). 
So, $ \|z_0-a_{x_0}\| < (t_0-1)\|x_0-w_0\| + \|x_0-w_0\| = t_0\|x_0-w_0\| = \|z_0-w_0\| \ldots \ldots (2) $ \\
Suppose, $z_0$, $x_0$ and $a_{x_0}$ are co-linear. 
As, $ z_0 = (1-t_0)w_0 + t_0x_0 $, so, $x_0$, $a_{x_0}$ and $w_0$ are co-linear and we know $ \|x_0-w_0\| = \|x_0 -a_{x_0}\| $. So, we must have $ w_0 = 2x_0-a_{x_0} $. 
Now,
\begin{eqnarray*}
    \|x_0-z_0\| + \|z_0-a_{x_0}\| & = & \|x_0-(1-t_0)w_0-t_0x_0\| + \|(1-t_0)w_0 + t_0x_0-a_{x_0}\| \\
		                              & = & (t_0-1)\|x_0-w_0\| + \|t_0x_0 + (1-t_0)(2x_0-a_{x_0})-a_{x_0}\| \\
																	& = & (t_0-1)\|x_0-w_0\| + (2-t_0)\|x_0-a_{x_0}\| \\
																	& = & (t_0-1)\|x_0-a_{x_0}\| + (2-t_0)\|x_0-a_{x_0}\| \\
																	& = & \|x_0-a_{x_0}\| \\
																	& = & \|x_0-w_0\| 
\end{eqnarray*}
\noindent So, $ \|x_0-w_0\| = \|x_0-z_0\| + \|z_0-a_{x_0}\| > \|z_0-a_{x_0}\| $    [ as $ x_0 \neq z_0 $ ] 
And we have, $ \|z_0-w_0\| = t_0\|x_0-w_0\| > \|x_0-w_0\| $    [ as $ t_0 > 1 $ ] 
From the above two inequality we have $ \|z_0-w_0\| > \|x_0-w_0\| > \|z_0-a_{x_0}\| \ldots \ldots (3)$.\\ 
So, from $ (1) $, $ (2) $ and $ (3) $ we have $ \|z_0-w_0\| > \|z_0-a\| $ for all $ a \in A $. 
As $ A $ is uniquely remotal so there exists $ a_{z_0} \in A $ such that $ \|z_0-w_0\| > \|z_0-a_{z_0}\| $, which contradicts the fact that $ D(z_0, \overline{A}) = D(z_0,A) $. 
So, our first assumption is wrong and $ \overline{A} $ is uniquely remotal. 
\end{proof}
\begin{theorem}
Let $ X $ be a finite dimensional strictly convex normed linear space. If $ A\subseteq X $ is uniquely remotal then $ A $ is singleton.
\end{theorem}
\begin{proof} By Theorem~\ref{theorem:closure}, $ \overline{A} $ is uniquely remotal. As $ X $ is finite dimensional, $ \overline{A} $ is compact. Then by Klee's result in \cite{K}, $ \overline{A} $ is a singleton, which implies that $ A $ is singleton. 
\end{proof}

In a normed linear space the closure of a M-compact set $A$ is M-compact but the converse is not true. Next we  show that the converse holds if the space is strictly convex and the set $A$ is uniquely remotal.

\begin{theorem}
Let $ X $ be a strictly convex normed linear space. If $ A \subseteq X $ be uniquely remotal and $ \overline{A} $ is M-compact then $ A $ is M-compact.
\end{theorem}
\begin{proof} Since $ X $ is a strictly convex normed linear space and $ A\subseteq X $ is uniquely remotal, $ \overline{A} $ is uniquely remotal, by Theorem~\ref{theorem:closure}. \\
Suppose $ A $ is not M-compact then there exists a maximizing sequence $ \{x_n\} \subseteq A $ such that $ \{x_n\} $ has no convergent subsequence. Since $ \{x_n\} $ is maximizing sequence in $ A $, there exists $ x_0 \in X $ such that $ \lim \limits_{n \to \infty} \|x_n-x_0\| = D(x_0, A).$ 
Now, $ D(x_0, A) = D(x_0, \overline{A}), $ so $ \{x_n\} $ is maximizing sequence for $ x_0 $ in $ \overline{A} $. Since $ \overline{A} $ is M-compact, $ \{x_n\} $ has a convergent subsequence $ \{x_{n_k}\} $ in $ \overline{A} $. \\
Suppose $ x_{n_k} \rightarrow w_0 \in \overline{A} \setminus A $ then $ \|x_0-w_0\| = \lim \limits_{k \to \infty}\|x_0-x_{n_k}\| = D(x_0, \overline{A}) $. Since $ A $ is uniquely remotal, there exists $ a_{x_0} \in A $ such that $ \|x_0-a_{x_0}\| = D(x_0, A). $ So we have $ \|x_0-a_{x_0}\| = D(x_0, A) = D(x_0, \overline{A}) = \|x_0-w_0\|, $ which contradicts the fact that $ \overline{A} $ is uniquely remotal. Thus $ A $ is M-compact. 
\end{proof}
We next show that if $ A$ is a uniquely remotal M-compact subset such that $ Far{A} \not\subseteq A'\ $ then $ A $ is singleton.
\begin{theorem}\label{theorem:singleton}
Let $ X $ be a normed linear space. If $ A \subseteq X $ is uniquely remotal and ${M}$-compact and $ Far{A} \not\subseteq A'\ $ then $ A $ is singleton. 
\end{theorem}
\begin{proof} As $ Far{A} \not\subseteq A'\ $. So there exists $a_0 \in Far{A}\setminus{A'}$. As $ A $ is uniquely remotal so there exists $ x \in X $ such that $ \|x-a_0\|=\sup \{\|x-a\|: a\in A\} $.
Consider the set $ C=\{0\leq{t}\leq{1} : a_0 $ is a farthest point for $ (1-t)a_0 + t x \} $ 
Clearly $ C\neq \phi $ as $ 1 \in C $. 

We claim that $ \inf C>0$ otherwise $ a_0 $ is farthest point of itself which implies $ A $ is singleton. 
Suppose $ \inf C=t_0>0 $. consider $ x_0=(1-t_0)a_0+t_0 x $. 
We claim that there exists $ \epsilon_0 > 0 $ such that $ \|x_0-a\| < \|x_0-a_0\|- \epsilon_0 ~~~ \forall ~~ a\in A\setminus\{a_0\} $. 

If not then there exists $ \{a_n\} \subseteq A\setminus\{a_0\} $ such that $ \|x_0-a_n\|\rightarrow \|x_0-a_0\|=D(x_0,A) $. 
So $ \{a_n\} $ is maximizing sequence for $ x_0 $. Since $ A $ is M-compact there exists $ \{a_{n_k}\} \subseteq \{a_n\} $ such that $ a_{n_k} \rightarrow a'. $  Since $ a_0 $ is not a limit point of $ A,$ $ a'\neq a_0. $
This contradicts the fact that $ A $ is uniquely remotal.

Now consider, $ \|(1-t)a_0 + t x_0 - a\| $  where $ a\in A\setminus\{a_0\} $ and $ t \in \left[0,1\right] $ 
Now,
\begin{eqnarray*}
 \|(1-t)a_0 + t x_0 - a\| & \leq & (1-t)\|a_0-a\| + t\|x_0-a\|  \\
                         & < & (1-t)  diam A + t[ \|x_0-a_0\| - \epsilon_0 ] \\
												 & = & t\|x_0-a_0\| - \epsilon_0 t + (1-t)diamA \\
												 & = & \|t x_0+ (1-t)a_0-a_0\| - [ \epsilon_0 t - (1-t) diam A]
\end{eqnarray*}
\noindent Now, consider the mapping	
\[ f: \mathbb{R} \rightarrow \mathbb{R}~~  \mbox{defined as}~~  f(t)= \epsilon_0 t - (1-t)diamA .\]
 
Clearly $ f $ is continuous and $ f(1)=\epsilon_0 > 0 $. So there exists $ \delta > 0 $ such that $ f(t)>0 $ for all $ t\in (1-\delta,1+\delta) $. So there exists at least one $ t' \in (1-\delta,1) $ such that $ f(t')>0 $ i.e. $ \epsilon_0 t'-(1-t')diamA > 0 $.
So $ \|(1-t')a_0 + t' x_0 -a\| < \|(1-t')a_0 + t' x_0 - a_0\| $ for all $ a\in A\setminus\{a_0\} $. 
So $ a_0 $ is a farthest point of $ (1-t')a_0 + t' x_0 $ in $ A $. Which contradicts the fact $ \inf C = t_0 >0 $. So $ inf C = 0 $ i.e. $ a_0 $ is a farhest point of itslf. So $ A $ is singleton. 
\end{proof}
Now we show that if $A$ is a uniquely remotal M-compact subset with $A'\neq \emptyset$ then $ A'$ is ${M}$-compact and uniquely remotal.
\begin{theorem}
Let $ X $ be a normed linear space. If $ A \subseteq X $ is uniquely remotal and ${M}$-compact then $ A'$ is either empty or $ A'$ is ${M}$-compact and uniquely remotal.
\end{theorem}
\begin{proof} If  $ Far{A} \not\subseteq A'\ $, then $ A $ is singleton by Theorem~\ref{theorem:singleton}.  Suppose $ Far{A} \subseteq A' $. We claim that $ D(x,A)=D(x,A') $. 
                     
										First we show that  $ D(x,A) \geq D(x,A') $ for all $ x \in X $.
For  $ x\in X $ there exists $ \{a'_n\}	\subseteq A' $ such that $ \|x-a'_n \| \rightarrow D(x,A') $ as $ n \rightarrow \infty $. As $ a'_n \in A'$, so for each $ n \in \mathbb{N} $	there exists $ a_n \in A $ such that
 $ \|a_n - a'_n\| < 1/n $.	Now, $\|x-a_n\| = \|x-a'_n-(a_n-a'_n)\| \geq \|x-a'_n\|-\|a_n-a'_n\| >\|x-a'_n\|-1/n $. \\
So, $ \lim \limits_{n \to \infty} \|x-a_n\| \geq D(x,A')	$. So, $ D(x,A) \geq \lim \limits_{n \to \infty} \|x-a_n\| \geq D(x,A') $ for all $ x \in X	$. 
	
We next show that  $ D(x,A') \geq D(x,A) $ for all $ x \in X $. Since $ A $ is uniquely remotal,  for each $ x \in X $ there exists $ a_x \in A $ such that $ \|x-a_x\|=D(x,A) $.  Then $ a_x \in A'$, since $ FarA \subseteq A'$. 
So, $ D(x,A') \geq \|x-a_x\| = D(x,A) \geq D(x,A') $. So, $ D(x,A')=D(x,A)	$.
 
We next prove that $ A'$ is M-compact. Let $ \{x_n\} \subseteq A'$ be a maximizing sequence in $ A'$, so there exists $ x \in X $ such that $ \lim \limits_{n \to \infty} \|x-x_n\| = D(x,A') = D(x,A) $. For each $ x_n \in A'$  there exists $ a_n \in A $ such that $ \|x_n - a_n\| < 1/n $ for all $ n \in \mathbb{N} $. 
Now, $ \|x-a_n\| = \|x-x_n+x_n-a_n\| \leq \|x-x_n\| + \|x_n-a_n\| < \|x-x_n\| + 1/n $. So, $ \lim \limits_{n \to \infty} \|x-a_n\| \leq \lim \limits_{n \to \infty} \|x-x_n\| $. 
Again, $ \|x-x_n\| = \|x-a_n+a_n-x_n\| \leq \|x-a_n\| + \|x_n-a_n\| < \|x-a_n\| + 1/n $. so, $ \lim \limits_{n \to \infty} \|x-x_n\| \leq \lim \limits_{n \to \infty} \|x-a_n\| $. 
So, $ \lim \limits_{n \to \infty} \|x-a_n\| = \lim \limits_{n \to \infty} \|x-x_n\| = D(x,A') = D(x,A) $. 
Thus $\{a_n\} $ is maximizing in $ A $ for $ x \in X $. As $ A $ is M-compact, so $ \{a_n\} $ has a convergent subsequence $ \{a_{n_k}\} $ converging to $a_0$ (say) in $A.$  Since $ \|x_n - a_n\| < 1/n $ for all $ n \in \mathbb{N} $, the subsequence $ \{ x_{n_k} \}$ of sequence $\{x_n\} $  converges to $a_0$ in $A \subset A'$. Thus every maximizing sequence in $ A'$ has a convergent subsequence and so $A'$ is M-compact.

Finally we prove that $ A'$ is uniquely remotal. As $ D(x,A) = D(x,A') $ and $ A $ is uniquely remotal, so $ A'$ is remotal. Suppose that $ A'$ is not uniquely remotal, then there exists $ x_0 \in X $ such that $ F(x_0,A') $ is not singleton. Let $ a_1, a_2 \in A'$ such that $ a_1, a_2 \in F(x_0,A') $. As $ a_1, a_2 \in A'$, so there exists $ \{x_n\} \subseteq A $ and $ \{y_n\} \subseteq A $ such that $ x_n \rightarrow a_1 $ and $ y_n \rightarrow a_2 $, as $ n \rightarrow \infty $. So, $ \lim \limits_{n \to \infty} \|x_0-x_n\| = \|x_0-a_1\| = \|x_0-a_2\| = \lim \limits_{n \to \infty} \|x_0-y_n\| = D(x_0,A) = D(x_0,A') $. So, $ \{x_n\} $ and $ \{y_n\} $ are maximizing sequence for $ x_0 $ in $ A $. Since$ A $ is M-compact $ \{x_n\} $ and $ \{y_n\} $ have convergent subsequences in $A$ which implies that $ a_1, a_2 \in A $ - this contradicts the fact that $ A $ is uniquely remotal. Thus $ A'$ is uniquely remotal. 
 
\end{proof}

Finally we show that if $A$ is a uniquely remotal M-compact set and $A'$ is compact then $A$ is singleton.

\begin{theorem}
Let $ X $ be a Banach space. If $ A \subseteq X $ is uniquely remotal and ${M}$-compact and $ A' $ is compact then $ A $ is singleton. 
\end{theorem}
\begin{proof}  To prove this theorem we need the following two lemmas: 
\begin{lemma}\label{lemma:M-compact}
Let $ X $ be a normed linear space. If $ A \subseteq X $ is uniquely remotal and M-compact then $ Far{A} = \{ a \in A : ~~ a $ is a farthest point for some $ x \in X \} $ is M-compact. 
\end{lemma}
\begin{proof} Let $ \{a_n\} $ be a maximizing sequence in $ Far{A} $. So there exists $ x \in X $ such that $ \lim \limits_{n \to \infty} \|x-a_n\|=D(x, Far{A})=D(x,A) $.
So $ \{a_n\} $ is maximizing sequence for $ x $ in $ A $. As $ A $ is M-compact, so $ \{a_n\} $ has a convergent 
subsequence $ \{a_{n_k}\} $. As $ A $ is uniquely remotal, so $ a_{n_k} \rightarrow a_x \in A $ (say).  
Now, $ \lim \limits_{k \to \infty} \|x-a_{n_k}\|=\|x-a_x\|= D(x,A)=D(x,Far{A}) $, so $ a_x \in FarA $. So $\{a_n\}$ has convergent sub sequence in $ Far{A} $.  Thus $Far{A}$ is M-compact.
\end{proof}
\begin{lemma}\label{lemma:remotal}
Let $ X $ be a normed linear space. If $ A\subseteq X $ is uniquely remotal and M-compact then $ \overline{A} $ is uniquely remotal. 
\end{lemma}
\begin{proof} As $ D(x,A) = D(x, \overline{A}) ~~ \forall x \in X $ and $ A $ is uniquely remotal so $ \overline{A} $ is remotal. Suppose $ \overline{A} $ is not uniquely remotal, then   there exists $ x_0 \in X $ such that $ F(x_0,\overline{A}) $ is not singleton.  Let $ a_1, a_2 \in \overline{A} $ such that $ a_1, a_2 \in F(x_0, \overline{A}). $ Since $ a_1, a_2 \in \overline{A} $ then there exists two sequences $ \{x_n\} \subseteq A ~~ and ~~ \{y_n\} \subseteq A $ such that $ x_n\rightarrow a_1 ~~ and ~~ y_n\rightarrow a_2 $. Clearly $ \{x_n\} ~~ and ~~ \{y_n\} $ are maximizing sequences for $ x_0 $ in $ A $, so $ \{x_n\} ~~ and ~~ \{y_n\} $ has convergent subsequences (as $ A $ is M-compact) which implies that $ a_1, a_2 \in A $ and this contradicts the fact that $ A $ is uniquely remotal. Thus $ \overline{A} $ is uniquely remotal. 
\end{proof}

\noindent\textbf{Proof of the theorem cont.} If $ Far{A} \not\subseteq A'\ $, then $ A $ is singleton by Theorem~\ref{theorem:singleton}.  Suppose $ Far{A} \subseteq A' $. As $ A' $ is closed then $ \overline {Far{A}} \subseteq A'. $ Since $ A' $ is compact, $ \overline{Far{A}} $ is also compact. Also, $ A $ is uniquely remotal and M-compact. So by Lemma~\ref{lemma:M-compact}, $ Far{A} $ is M-compact. Also from definition $ Far{A} $ is uniquely remotal. Now by Lemma~\ref{lemma:remotal}, $ \overline{Far{A}} $ is uniquely remotal. As $ Far{A} $ is compact and uniquely remotal so $ Far{A} $ is singleton by Klee's result in \cite{K}. Then $ A $ is singleton.

\end{proof}

\bibliographystyle{amsplain}

\end{document}